\documentclass[12pt]{amsart}
\usepackage{amscd}
\usepackage{amssymb}
\usepackage{amstext}
\usepackage{amsthm}
\numberwithin{equation}{section}
\usepackage{latexsym}
\usepackage{xcolor}
\numberwithin{equation}{section}

\DeclareMathOperator{\Hol}{Hol}
\DeclareMathOperator{\Span}{Span}
\DeclareMathOperator{\Ker}{Ker}

\renewcommand{\phi}{\varphi}

\newcommand{\clos}{{\rm Clos}\,}

\newtheorem{Thm}{Theorem}[section]
\newtheorem{theorem}[Thm]{Theorem}
\newtheorem{lemma}[Thm]{Lemma}
\newtheorem{claim}[Thm]{Claim}

\newtheorem{corollary}[Thm]{Corollary}

\newcommand{\FF}{\mathcal{F}}
\newcommand{\EE}{\mathcal{E}}
\newcommand{\PP}{\mathcal{P}}
\newcommand{\ZZ}{\mathcal{Z}}
\newcommand{\TT}{\mathcal{T}}
\newcommand{\CCC}{\mathcal{C}}
\newcommand{\CC}{\mathbb{C}}
\newcommand{\ve}{{\varepsilon}}
\DeclareMathOperator{\dist}{dist}
\DeclareMathOperator{\card}{card}
\DeclareMathOperator{\supp}{supp}
\DeclareMathOperator{\osc}{osc}

\begin{document}
\sloppy
\title[The Newman--Shapiro problem]
{The Newman--Shapiro problem}

\author{Yurii Belov, Alexander Borichev}
\address{
%Anton Baranov,
%\newline Department of Mathematics and Mechanics, St.~Petersburg State University, St.~Petersburg, Russia,
%\newline National Research University Higher School of Economics, St.~Petersburg, Russia,
%\newline {\tt anton.d.baranov@gmail.com}
%\smallskip
%\newline \phantom{x}\,\, 
Yurii Belov,
\newline St.~Petersburg State University, St. Petersburg, Russia,
\newline {\tt j\_b\_juri\_belov@mail.ru}
\smallskip
\newline \phantom{x}\,\, Alexander Borichev,
\newline I2M, CNRS, Aix-Marseille Universit\'e, 13453 Marseille, France, 
\newline {\tt alexander.borichev@math.cnrs.fr}
}
\thanks{This work was supported by Russian Science Foundation grant 17-11-01064.}

\begin{abstract} We give a negative answer to the Newman--Shapiro problem on weighted approximation for entire functions formulated in 1966 
and motivated by the theory of operators on the Fock space. There exists a function in the Fock space such that its exponential multiples do not approximate 
some entire multiples in the space. 
Furthermore, we establish several positive results under different restrictions on the function in question.
\end{abstract}

\maketitle

\section{Introduction and the main results}

Let $\FF=\FF_{(1)}$ be the classical Bargmann--Segal--Fock space, where
$$
\mathcal F_{(\alpha)}=\Big\{F\in\Hol(\CC):\|F\|^2_\mathcal F=\frac1\pi\int_{\CC}
|F(z)|^2e^{-{\alpha\pi}|z|^2}\,dm(z)<\infty\Big\},
$$ 
and $m$ stands for the area Lebesgue measure. This space serves as a model of the phase space of a
particle in quantum mechanics and so plays an important role in theoretical physics. Moreover, this space appears in time-frequency
analysis, as a spectral model of $L^2(\mathbb{R})$ 
via the Bargmann transform (see, e.g., \cite[Section 3.4]{Gro}). 
Note also that the complex exponentials
$e_\lambda$, $e_\lambda(z)=e^{\lambda z}$ are the reproducing kernels of $\mathcal{F}$, i.e., 
$$
\langle F,k_\lambda\rangle_{\mathcal{F}}=F(\lambda), \qquad F\in\mathcal{F},
$$
where $k_{\lambda}=\pi e_{\pi\bar{\lambda}}$.
\smallskip

In 1966, D.~J.~Newman and H.~S.~Shapiro posed in \cite{NS} the following problem about 
the structure of the operator adjoint to the multiplication
operator in Fock space. Let $F$ be an entire function such that, for every $A>0$,
\begin{equation}
|F(z)|\leq C(A,F)\exp\Bigl(\frac\pi2|z|^2-A|z|\Bigr),\qquad z\in\mathbb C.
\label{expest}
\end{equation}
This condition is equivalent to the following one:
$e_\lambda \cdot F \in\mathcal{F}$ for every $\lambda\in\mathbb{C}$. Now we can define 
the multiplication operator $M_F:G\mapsto FG$ %at least 
on the 
linear span of the exponentials
$$
\mathcal{L}=\Span\{e_\lambda : \lambda\in\mathbb{C}\}.
$$
The natural domain of the operator $M_F$ is given by
$$
\mathcal{R}_F=\{G\in\mathcal{F}: FG\in\mathcal{F}\}.
$$
Thus, we can consider the adjoint operator $M^*_F$ as well as the operator adjoint to the restriction $M_F\bigl{|}_{\mathcal{L}}$, which we
(following \cite{NS}) denote by $F^*\bigl{(}\frac{d}{dz}\bigr{)}$. This notation is motivated by the fact that when $F=P$ is a polynomial, we have
$$
P(\lambda)G(\lambda)=\langle M_PG, \pi e_{\pi\bar{\lambda}}\rangle =
\langle G, P^*(d\slash{dz})(\pi e_{\pi\bar{\lambda}})\rangle,
$$
where $P^*(z)=\overline{P(\bar{z})}$ and 
$P^*(\frac{d}{dz})$ is understood in the usual sense as a differential 
operator. In this case it is easy to see that 
$M^*_P=P^*\bigl{(}\frac{d}{dz}\bigr{)}$. The Newman--Shapiro 
problem (related to a much earlier work of E.~Fischer \cite{Fisch}) 
is whether $M^*_F=F^*\bigl{(}\frac{d}{dz}\bigr{)}$ 
%in the general case 
for all $F$ satisfying \eqref{expest}. In \cite{NS} (see also \cite{NS1} and an extended unpublished manuscript \cite{NS2}) Newman and Shapiro proved that this is the case when $F$ is an exponential polynomial (i.e.,  
$F=\sum_{k=1}^nP_ke_{\lambda_k}$, where $P_k$ are polynomials and $\lambda_k\in\mathbb{C}$) and for some other special cases (i.e. $F$ has no zeros or $F(z)=\sin z\slash z$). Moreover, they revealed some connections of this problem with  weighted polynomial approximation in $\mathcal{F}$. More precisely, they proved the following result (to avoid inessential technicalities we assume that $F$ has simple zeros only). 
Denote by $\mathcal E$ the space of all entire functions.

\begin{theorem}[{\cite[Theorem 1]{NS}}, \cite{NS1}]
For every $F$ satisfying estimates \eqref{expest} the following statements are equivalent:
\begin{enumerate}
\begin{item}
$\overline{\Span}\{z^nF, n\geq0\}=\mathcal E F\cap \mathcal F$;
\end{item}
\begin{item}
$M^*_F=F^*\bigl{(}\frac{d}{dz}\bigr{)}$;
\end{item}
\begin{item}
$\Ker F^*\bigl(\frac{d}{dz}\bigr)=\overline{\Span}\bigl\{e_{\bar{\lambda}}: e_{\bar{\lambda}}\in 
\Ker F^*\bigl(\frac{d}{dz}\bigr)\bigr\}=\overline{\Span}\bigl\{e_{\bar{\lambda}}: F(\lambda)=0\bigr\}.$
\end{item}
\end{enumerate}
\label{NSTh}
\end{theorem} 

%Newman and Shapiro established that the equivalent conditions of Theorem~\ref{NSTh} hold for exponential polynomials. 
%Furthermore, they posed the problem whether the equivalent %conditions of Theorem~\ref{NSTh} hold . 
The Newman--Shapiro problem remained open since 1966. Several similar questions were studied, e.g.,
in \cite{M} (see also \cite[Chapter X.8]{Eren}). For related questions on Toeplitz operators on the Fock space see \cite{CS} 
and the references therein.

It should be mentioned that the Newman--Shapiro problem is closely related to 
the spectral synthesis (hereditary completeness) problem for systems 
of reproducing kernels in the Fock space (or of Gabor-type
expansions with respect to time-frequency shifts of the Gaussian). 
In the Paley--Wiener space setting, the spectral synthesis problem 
was solved in \cite{BBB}, whereas for the reproducing kernels of the Fock 
space the solution (in general, also negative) was recently given in \cite{BBB1}.

In this article we prove that the answer to the Newman--Shapiro problem is in general negative and establish several positive results under different restrictions on the growth and regularity of the function $F$. 

The original Newman--Shapiro problem is formulated for the Fock spaces on $\mathbb C^n$, $n\ge 1$. 
Here, we restrict ourselves to the case $n=1$. 
The negative answer to the Newman--Shapiro problem in the case $n=1$ means the negative answer for every $n\ge 1$.
It seems plausible that one should use different techniques to obtain positive results in the case $n>1$.

\begin{theorem} For any $\alpha\in(1,2)$, there exist two entire functions 
$F$ and $G$ such that $G,GF\in \mathcal F$ and for every entire function $h$ of order 
at most $\alpha$ we have $hF\in \mathcal F$, but 
$$
GF\notin \overline{\Span}\bigl\{pF:\, p\in\mathcal P\bigr\} = \overline{\Span}\bigl\{e_\lambda F: \lambda\in\CC\bigr\}.
$$
\label{cexTh}
\end{theorem}

Thus, the equivalent conditions of Theorem~\ref{NSTh} do not hold for $F$.  
Here and later on, $\PP$ is the space of the polynomials. 

Next we establish that under more restrictive growth and regularity conditions on 
the function $F$ the answer to the Newman--Shapiro problem becomes positive. 

%Denote by $\EE$ the set of all entire functions. 
Given $\alpha\ge 0$, denote by $\EE_{2,\alpha}$ 
the class of all entire functions of type at most $\alpha$ for order $2$, that is %$F\in\EE$ is in $\EE_{2,\alpha}$ if and only if 
$$
\limsup_{|z|\to\infty}\frac{\log|F(z)|}{|z|^2}\le\alpha.
$$
Set $\EE_2=\cup_{\alpha<\infty}\EE_{2,\alpha}$.

Given $F\in\EE_2$, consider its indicator function for order $2$, 
$$
h_F(\theta)=\limsup_{r\to\infty}\frac{\log|F(re^{i\theta})|}{r^2},\qquad \theta\in[0,2\pi].
$$
We say that $F\in\EE_2$ is of completely regular growth if $\log|F(re^{i\theta})|/r^2$ converges uniformly in $\theta\in[0,2\pi]$ to $h_F(\theta)$ 
as $r\to\infty$ and $r\not\in E_F$ for some set $E_F\subset[0,\infty)$ of zero relative measure, that is 
$$
\lim_{R\to\infty}\frac{E_F\cap[0,R]}{R}=0.
$$

\begin{theorem} Let $F\in \EE$. Suppose that there exist $G\in\EE_2$ of completely regular growth 
and $\alpha<1$ such that $(FG\cdot \EE)\cap \FF_{(\alpha)}=FG\cdot\mathbb C$, and $\inf_{[0,2\pi]}h_G>0$.
Then $F\in \FF_{(\gamma)}$ for every $\gamma\ge \alpha$, and 
$$
\overline{\Span}\bigl\{e_\lambda F: \lambda\in\mathbb C\bigr\}=\EE F\cap \FF.
$$
\label{mainthm2}
\end{theorem}

Thus, the equivalent conditions of Theorem~\ref{NSTh} hold for such $F$.  

The conditions of the theorem mean that the zero set of $F$ can be complemented by a set of positive angular density to a set $\Lambda$ such that the 
system $\{k_\lambda\}_{\lambda\in \Lambda}$ is complete and minimal in $\FF_{(\alpha)}$. 

When the zero set of $F$ is sufficiently regular and not very dense, we get the following result.

\begin{corollary} Let $F\in \FF$ be of completely regular growth. Suppose that the upper Beurling--Landau density 
$D^+_{Z(F)}$ of the zero set $Z(F)$ of $F$ (with multiplicities taken into account) is less than $1/\pi$:
\begin{equation}
\limsup_{R\to\infty} \sup_{z\in\CC}\frac{\card(Z(f)\cap D(z,R))}{\pi R^2}<\frac1\pi.
\label{zve}
\end{equation}
Then  
$$
\overline{\Span}\bigl\{e_\lambda F: \lambda\in\mathbb C\bigr\}=\EE F\cap \FF.
$$
\label{cor}
\end{corollary}

Here and later on $D(z,r)$ stands for the open disc centered at $z$ of radius $r$. 

Condition \eqref{zve} is indispensable here as demonstrates the example given in the proof of Theorem~\ref{cexTh}.

When we restrict the growth of $F$, there are no more regularity restrictions on the zeros: 

\begin{theorem} There exists $\eta>0$ such that if $F\in \EE_{2,\eta}$, then  
$$
\overline{\Span}\bigl\{e_\lambda F: \lambda\in\mathbb C\bigr\}=\EE F\cap \FF.
$$
\label{thm2}
\end{theorem}

Thus, the situation here could be compared to that of the cyclicity/invertibility problem in the Bergman space. 
Invertibility does not imply cyclicity there \cite{BH}; if we impose additional growth restrictions, then invertibility 
does imply cyclicity. (Stronger lower estimates also imply cyclicity \cite{BI}). 
The main difference is that the Bergman space cyclic/invertible functions $f$ are zero free and one works with harmonic $\log|f|$ 
while in our situation the Fock space functions have a lot of zeros which makes the problem much more complicated.

The Fock space does not possess a Riesz basis of reproducing kernels. Instead, we have the system $\mathcal K=\{k_w\}_{w\in \ZZ_0}$ which is complete and minimal in $\mathcal F$. Here and later on $\ZZ=\mathbb Z+i\mathbb Z\subset\mathbb C$, $\ZZ_0=\ZZ\setminus\{0\}$. 
Let $\sigma$ be the Weierstrass sigma function associated to $\ZZ$, $\sigma_0(z)=\sigma(z)/z$. For more information about these functions see Section~\ref{sect2}. The system 
$\{g_w\}_{w\in \ZZ_0}$, $g_w=\sigma_0/(\sigma_0'(w)(\cdot-w))$, is biorthogonal to $\mathcal K$. One of our main technical tools to get 
the completeness results is the following Parseval-type relation:
if $F_1,F_2\in\FF$, $\mu\in Z(F_2)\setminus \ZZ_0$, then
\begin{multline*}
\sum_{w\in \ZZ_0}\langle F_2,k_w\rangle\cdot \langle g_w,F_1\rangle\Bigl[ \frac1{z-w}+\frac1{w-\mu} \Bigr] \\
= \frac{\langle F_2,F_1\rangle}{z}
+\Bigl\langle \frac{F_2}{\cdot-\mu}, F_1 \Bigr\rangle+o(|z|^{-1}),\qquad |z|\to\infty,\,z\in\CC\setminus\Omega, 
\end{multline*}
for some thin set $\Omega$. 
Furthermore, we study related continuous Cauchy transforms corresponding to pairs of Fock space functions, whose asymptotics gives their scalar product. 
In particular, given $F_1,F_2\in\FF$, we have
\begin{multline*}
\frac{1}{\sigma_0(z)}\Bigl\langle \frac{\sigma_0(z)F_1-F_1(z)\sigma_0}{z-\cdot},F_2\Bigr\rangle\\=
%\frac{1}{\sigma_0(z)}\int_\CC \frac{\sigma_0(z)F_2(\zeta)-\sigma_0(\zeta)F_2(z)}{z-\zeta}\overline{F_1(\zeta)}\,e^{-\pi|\zeta|^2}\,dm_2(\zeta)\\=
\int_\CC \frac{F_2(\zeta)\overline{F_1(\zeta)}}{z-\zeta}\,e^{-\pi|\zeta|^2}\,dm_2(\zeta)-
\frac{F_2(z)}{\sigma_0(z)}\int_\CC \frac{\sigma_0(\zeta)\overline{F_1(\zeta)}}{z-\zeta}\,e^{-\pi|\zeta|^2}\,dm_2(\zeta)\\=\frac{\langle F_2,F_1\rangle}{z}+o(|z|^{-1}),\qquad |z|\to\infty,\,z\in\CC\setminus\Omega, 
\end{multline*}
for some thin set $\Omega$. Finally, we establish and use a number of uniqueness results on the Fock space functions outside thin sets (thin lattice sets). 

The plan of the paper is as follows. In Section~\ref{sect2} we introduce some notations and prove three uniqueness results for functions in the Fock space. 
Section~\ref{sect3} contains several auxiliary results on interpolation formulas and the scalar product in the Fock space. 
Theorem~\ref{mainthm2} together with some auxiliary lemmas is proved in Section~\ref{sect4}.  
Theorem~\ref{thm2} and Corollary~\ref{cor} are proved in Section~\ref{sect5}. Finally, Theorem~\ref{cexTh} is proved in 
Section~\ref{sect6} using techniques that are quite different from those in the previous part of the paper.

\subsection*{Notations.} Throughout this paper the notation $U(x)\lesssim V(x)$ means that there is a constant $C$ such that
$U(x)\leq CV(x)$ holds for all $x$ in the set in question, $U, V\geq 0$. We write $U(x)\asymp V(x)$ if both $U(x)\lesssim V(x)$ and
$V(x)\lesssim U(x)$.

\subsubsection*{\bf Acknowledgments}
We thank Anton Baranov for numerous useful discussions.
%\newpage 

\section{Notations and some uniqueness results for the Fock space}
\label{sect2}

In this section after introducing some notations, we establish three uniqueness results for the Fock space functions. 

Given $\alpha\in\CC$, the time-frequency shift operator $\TT_\alpha$ given by 
$$
(\TT_\alpha F)(z)=e^{\pi\bar\alpha z-\frac\pi2|\alpha|^2}F(z-\alpha)
$$
is unitary on the Fock space $\FF$.

Put $d\nu(z)=e^{-\pi|z|^2}\,dm_2(z)$.

Given $F\in\EE$ we denote by $Z(F)$ its zero set. 

It is known \cite[Theorem 5, Chapter 3]{Lev} that if $F,G\in\EE_2$, $F$ is of completely regular growth, then 
$$
h_{FG}=h_F+h_G.
$$

Together with $\FF$ we consider its subspace 
$$
\FF_0=\{F\in\EE:\PP F\in\FF\}.
$$
Given $F\in\FF_0$, denote
$$
[F]_\FF=\overline{\Span}\{\PP F\}.
$$

Following \cite{BBB1} we say that a measurable subset of $\CC$ is thin if it is the union of a measurable set $\Omega_1$ of zero (area) density,
$$
\lim_{R\to\infty}\frac{m_2(\Omega_1\cap D(0,R))}{R^2}=0,
$$
and a measurable set $\Omega_2$ such that 
$$
\int_{\Omega_2}\frac{dm_2(z)}{(|z|+1)^2\log(|z|+2)}<\infty.
$$

The union of two thin sets is thin. If $\Omega$ is thin, then its lower density 
$$
\liminf_{R\to\infty}\frac{m_2(\Omega\cap D(0,R))}{R^2}
$$
is zero. In particular, $\CC$ is not thin. If $\Omega$ is thin, then its translations $z+\Omega$ are thin, $z\in\CC$.

We start with the following Liouville type result. Although we do not use it directly in the paper, it helps us to better understand how sparse thin sets are with respect to the small value %level 
sets of the Fock space functions. The lemma was originally formulated in \cite[Lemma 4.2]{BBB1}. A corrected proof is given in \cite{BBB1a}. Here we give an alternative proof. 

\begin{lemma} Let $F$ be an entire function of finite order, bounded on $\CC\setminus \Omega$ for some thin set $\Omega$. 
Then $F$ is a constant. 
\label{L40}
\end{lemma}

\begin{proof}
Suppose that $F$ is not a constant and that 
\begin{equation}
\log|F(z)|=O(|z|^N),\qquad |z| \to\infty,
\label{eq1}
\end{equation}
for some $N<\infty$.
We can find $w\in\CC$ and $c\in\mathbb R$ such that the subharmonic function $u$,
$$
u(z)=\log|F(z-w)|+c
$$
is negative on $\CC\setminus \widetilde{\Omega}$ for some open  
thin set $\widetilde{\Omega}$, and $u(0)=1$. 
Given $R>0$, consider the connected component $O^R$ of $\widetilde{\Omega}\cap D(0,R)$ containing the point $0$. 
By the theorem on harmonic estimation \cite[VII.B.1]{Koo1}, we have
$$%\begin{equation}
1=u(0)\le \omega(0,\partial O^R\cap \partial D(0,R),O^R)\cdot \max_{|z|=R}u(z),
%\label{eq2}
$$%\end{equation}
where $\omega(z,E,O)$ is the harmonic measure at $z\in O$ of $E\subset \partial O$ with respect to the domain $O$. 
By \eqref{eq1} we obtain 
\begin{equation}
\omega(0,\partial O^R\cap \partial D(0,R),O^R)\ge aR^{-N},\qquad R\ge 1,
\label{dd4}
\end{equation}
for some $a>0$.

For some $R\ge 1$ to be chosen later on we set 
\begin{align*}
\phi(z)&=\begin{cases}
\omega(z,\partial O^R\cap \partial D(0,R),O^R),\qquad z\in O^R,\\
1,\qquad z\in \partial O^R\cap \partial D(0,R),\\
0,\qquad z\in (D(0,R)\setminus O^R)\cup(\partial D(0,R)\setminus \partial O^R),
\end{cases}\\
\psi(r)&=\max_{\partial D(0,r)}\phi.
\end{align*}

By the maximum principle, $\psi$ increases on $[0,R]$. 

We use the following radial version of Hall's lemma (attributed to {\O}ksendal in \cite[p.125]{GM}): if $E$ is a measurable subset of 
$D(0,1)\setminus D(0,1/2)$, then 
$$
\omega(0,E,D(0,1)\setminus E)\ge \delta m_2(E)
$$
for some absolute constant $\delta>0$.

Let $0<r<(1+\varepsilon)r<R$ for some $\varepsilon\in(0,1/2)$ and assume that 
$$
m_2\bigl(O^R\cap D(0,(1+\varepsilon)r)\setminus D(0,r)\bigr)\le\frac{\pi\varepsilon^2}{8}r^2.
$$
Then by Hall's lemma, applied in the discs 
$$
D(\zeta,\varepsilon r/2),\qquad \zeta\in\partial D(0,(1+\varepsilon/2)r), 
$$
we obtain 
$$
\psi(r)\le (1-\beta)\psi((1+\varepsilon)r),
$$
for some absolute constant $\beta>0$. Choose $\varepsilon>0$ in such a way that $(1+\varepsilon)^{2N}(1-\beta)=1$ and assume that $R=(1+\varepsilon)^M$ for some 
integer $M$. Put  
$$
\mathcal N=\Bigl\{n\ge 0: m_2\bigl(\widetilde{\Omega}\cap D(0,(1+\varepsilon)^{n+1})\setminus D(0,(1+\varepsilon)^n)\bigr)>\frac{\pi\varepsilon^2}{8}(1+\varepsilon)^{2n} \Bigr\},
$$
and set $\mathcal N^*_M=\mathbb Z_+\cap[0,M)\setminus \mathcal N$. 

Then 
$$
\psi((1+\varepsilon)^n)\le (1+\varepsilon)^{-2N}\psi((1+\varepsilon)^{n+1}),\qquad n\in\mathcal N^*_M.
$$
By \eqref{dd4} we obtain that 
$$%\begin{multline*}
a(1+\varepsilon)^{-NM}\le \psi(1)\le (1+\varepsilon)^{-2N\card(\mathcal N^*_M)}\psi(R)=(1+\varepsilon)^{-2N\card(\mathcal N^*_M)},
$$%\end{multline*}
and, hence,
$$
M\ge 2\card(\mathcal N^*_M)-c,\qquad M\ge 0.
$$
In particular,
\begin{equation}
\card([3^s,3^{s+1})\cap\mathcal N)\ge 3^{s-2},\qquad s\ge s_0.
\label{dd7}
\end{equation}

We have $\widetilde{\Omega}=\Omega_1\cup\Omega_2$, where $\Omega_1$ and $\Omega_2$ 
are open, and
\begin{align*}%\label{eqstar}
\lim_{R\to\infty}\frac{m_2(\Omega_1\cap D(0,R))}{R^2}=0,\\
%\label{eqstarstar}
\int_{\Omega_2}\frac{dm_2(z)}{(|z|+1)^2\log(|z|+2)}<\infty.
\end{align*}
Furthermore,
$$
\int_{\Omega_{2,n}}\frac{dm_2(z)}{(|z|+1)^2\log(|z|+2)}\ge \frac{c}{n},\qquad 
n\ge n_0,\,n\in\mathcal N,
$$
for some $c>0$, where $\Omega_{2,n}=\Omega_2\cap D(0,(1+\varepsilon)^{n+1})\setminus D(0,(1+\varepsilon)^n)$.
Thus,
$$%\begin{equation}
\sum_{n\in\mathcal N}\frac{1}{n}<\infty,
$$%\end{equation}
that contradicts to \eqref{dd7}. This completes the proof.
\end{proof}

We say that a subset $A$ of the lattice $\ZZ=\mathbb Z+i\mathbb Z$ is {\it lattice thin} if for some (every) $c>0$, the set 
$$
\cup_{w\in\ZZ}D(w,c)
$$
is thin.

Let $\sigma$ be the Weierstrass sigma function associated to $\ZZ$,
$$
\sigma(z)=z\prod_{w\in\ZZ_0}\Bigl(1-\frac zw\Bigr)e^{\frac zw+\frac{z^2}{2w^2}},
$$
where $\ZZ_0=\ZZ\setminus\{0\}$. Set 
$\sigma_0(z)=\sigma(z)/z$. 
Since 
$$
|\sigma(z)|\asymp \dist(z,\ZZ)e^{(\pi/2)|z|^2},\qquad z\in\mathbb C,
$$ 
we have $\sigma\in \EE_{2,\pi/2}$, $h_\sigma\equiv \pi/2$ and
$\ZZ_0$ is a uniqueness set for $\FF$.

\begin{lemma} There exists $\beta\in(0,1)$ such that if $S\in\mathcal E_2$ and $\mathcal Z_1$ is a subset of $\ZZ_0$ 
of lower density at least $1-\beta$ 
satisfying the property
$$
\inf_{\partial D(z,\rho\log^{-2}(1+|z|))}|S|<1, 
$$
for every $\rho\in(0,1)$, $z\in \mathcal Z_1$, 
then $S$ is a constant.
\label{L24}
\end{lemma}

\begin{proof} 
For every $z\in\mathcal Z_1$ put 
$$
\Delta_z=\bigl\{ w\in D(z,\log^{-2}(1+|z|)):|S(w)|<1 \bigr\}.
$$
Then every $\Delta_z$ contains a finite family of intervals $z+e^{i\theta_{k,z}}J^k_z$ with disjoint $J^k_z\subset \mathbb R_+$ of total length $(1/2)\log^{-2}(1+|z|)$. 

Set
$$
\Omega=\mathbb C\setminus \bigcup_{z\in\mathcal Z_1}\overline{\Delta_z}.
$$
Given $z\in\CC$ and $a>0$, set 
$$
\mathcal Z^{z,a}_*=\bigl\{ w\in\mathcal Z_1: \overline{\Delta_w}\subset D(z,a)\setminus D(z,a/2) \bigr\}.
$$

Next, given $\delta\in(0,1)$ to be chosen later on, if the lower density $\mathcal Z_1$ is at least $1-\beta$ with $0<\beta\le \beta(\delta)<1$, then
\begin{equation}
\card(\mathcal Z^{z,r\delta}_*)\ge \frac\pi4 \delta^2r^2,\qquad z\in\partial D(0,r),\, r>r(\delta).
\label{nes4}
\end{equation}

Now we are going to prove that, under condition \eqref{nes4}, we have  
\begin{equation}
\omega\bigl(z,\partial\Omega\cap D(z,\delta|z|),D(z,\delta|z|)\cap\Omega\bigr)\ge  \gamma,\qquad |z|>r(\delta),
\label{nes2}
\end{equation}
for some absolute constant $\gamma>0$.

Given $z\in\CC$, set $t=\delta|z|$,
\begin{gather*}
F=\bigl\{ w\in D(0,1): z+wt\in\cup_{w\in \mathcal Z^{z,r\delta}_*} \overline{\Delta_w}\bigr\},\\
\mu=c\sum_{w\in \mathcal Z^{z,r\delta}_*}\mu_w=c\sum_{w\in \mathcal Z^{z,r\delta}_*}\sum_k \chi_{(w+e^{i\theta_{k,w}}J^k_w-z)/t}\,m,
\end{gather*}
where $m$ is one dimensional Lebesgue measure and $c$ is a normalization constant such that $\mu$ is a probability measure.

Then, under condition \eqref{nes4}, we have $c\asymp \log^2t/t$ for large $t$, and the logarithmic energy of $\mu$ is estimated below as follows: 
\begin{multline*}
-I(\mu)=-\int\int\log|\zeta_1-\zeta_2|\,d\mu(\zeta_1)\,d\mu(\zeta_2)
=-c^2\sum_{w\in \mathcal Z^{z,r\delta}_*}I(\mu_w)\\-c^2
\sum_{w\in \mathcal Z^{z,r\delta}_*}\sum_{w_1\in \mathcal Z^{z,r\delta}_*\setminus\{w\}}\int\int\log|\zeta_1-\zeta_2|\,d\mu_w(\zeta_1)\,d\mu_{w_1}(\zeta_2)
\\ \le O(1)+\frac{\log^2t}{t}\cdot \sup_{u\in D(0,1)}\sum_{w\in \mathcal Z^{z,r\delta}_*}\int\log\frac1{|\zeta-u|}\,d\mu_w(\zeta)=O(1).
\end{multline*}
Since $\supp \mu\subset F$, the logarithmic capacity of $F$ is bounded below by an absolute constant $c>0$. 
Finally, \cite[Theorem III.9.1]{GM} yields \eqref{nes2}.

Put  
$$
\psi(r)=\max_{\partial D(0,r)}\log|S|,\qquad r>0.
$$
Since $S\in\mathcal E_2$, we have 
\begin{equation}
\psi(r)=O(r^2),\qquad r\to\infty.
\label{nes3}
\end{equation}
Under condition \eqref{nes4}, by the theorem on harmonic estimation \cite[VII.B.1]{Koo1} and by \eqref{nes2}, we obtain
\begin{equation}
\psi(r)\le \psi(r+\delta r)(1-\gamma),\qquad r>r(\delta).
\label{nes1}
\end{equation}
If $\delta$ is sufficiently small, $0<\delta\le (1-\gamma)^{-1/2}-1$, then \eqref{nes3} and \eqref{nes1} imply together that $\psi\le 0$ and, hence, 
$S$ is a constant. This completes the proof. 
\end{proof}

\begin{lemma} Let $F\in\FF_0\cap \ell^\infty(\ZZ)$. Then $F$ is a constant.
\label{L22}
\end{lemma}

\begin{proof} By the Lagrange interpolation formula, for every $k\ge 0$, $z\in\CC\setminus\ZZ_0$ we have 
$$
\frac{z^kF(z)}{\sigma(z)}=\sum_{w\in\ZZ}\frac{w^kF(w)}{\sigma'(w)(z-w)}
$$
and, hence,
$$
\Bigl|\frac{z^kF(z)}{\sigma(z)}\Bigr|=\Bigl|\sum_{w\in\ZZ}\frac{w^kF(w)}{\sigma'(w)(z-w)}\Bigr|\le 
\sum_{w\in\ZZ}\frac{|w|^k\cdot|F(w)|}{|\sigma'(w)|\cdot|z-w|}.
$$
Therefore,
$$
|F(z)|\lesssim |\sigma(z)|\cdot \min_{k\ge 0}\Bigl[\frac1{|z|^k} \sum_{w\in\ZZ}\frac{|w|^k}{|\sigma'(w)|} \Bigr],\qquad \dist(z,\ZZ)>\frac13.
$$
Thus,
\begin{multline*}
|F(z)|\lesssim e^{\frac\pi2|z|^2}\cdot \min_{k\ge 0}\Bigl[\frac1{|z|^k} \sum_{w\in\ZZ} |w|^k e^{-\frac\pi2|w|^2}\Bigr]\\ 
\asymp e^{\frac\pi2|z|^2}\cdot \min_{k\ge 0}\Bigl[\frac1{|z|^k} \int_{\mathbb C} |w|^{k}e^{-\frac\pi2|w|^2}\,dm_2(w)\Bigr]\\
=2\pi e^{\frac\pi2|z|^2}\cdot \min_{k\ge 0}\Bigl[\frac1{|z|^k} \int_0^\infty r^{k+1}e^{-\frac\pi2r^2}\,dr\Bigr]\\
\lesssim \min_{k\ge 0}\exp\Bigl[ \frac\pi2|z|^2-k\log|z|+\frac{k+1}2\log\frac{k+1}\pi-\frac{k+1}2\Bigr]\\
\lesssim 1+|z|,\qquad \dist(z,\ZZ)>\frac13. 
\end{multline*}
It remains to use the Liouville theorem.
\end{proof}

\section{Interpolation formulas and duality in the Fock space}
\label{sect3}

In this section we establish several results on relations between interpolation formulas, expansions with respect to some fixed 
complete and minimal systems of the reproducing kernels and their biorthogonal systems, and the scalar product in the Fock space.

\begin{lemma} Let $F_1,F_2\in\FF$, $F_3\in\FF_0$. Then 
\begin{align*}
\Bigl|\int_\CC \frac{F_2(\zeta)\overline{F_1(\zeta)}}{z-\zeta}\,d\nu(\zeta) \Bigr|&=o(1),\qquad |z|\to\infty,\\
\Bigl|\int_\CC \frac{F_3(\zeta)\overline{F_1(\zeta)}}{z-\zeta}\,d\nu(\zeta) \Bigr|&=O((1+|z|)^{-1}),\qquad |z|\to\infty,\\
\Bigl|\int_\CC \frac{\sigma_0(\zeta)\overline{F_1(\zeta)}}{z-\zeta}\,d\nu(\zeta) \Bigr|&=O(|z|^{-1}\log^{1/2}|z|),\qquad |z|\to\infty.
\end{align*}
\label{LA}
\end{lemma}

\begin{proof} We use that if $F\in\FF$, then $|F(z)|=o(e^{\pi|z|^2/2})$, $|z|\to\infty$. Furthermore, 
$F_2(\zeta)\overline{F_1(\zeta)}\,d\nu(\zeta)=\phi(\zeta)\,dm_2(\zeta)$ 
with $\phi\in L^1(\CC)\cap C_0(\CC)$.  
Therefore, for every $R>0$, we have
\begin{multline*}
\Bigl|\int_\CC \frac{F_2(\zeta)\overline{F_1(\zeta)}}{z-\zeta}\,d\nu(\zeta) \Bigr| \\ 
\le \Bigl| \int_{\CC\setminus D(z,R)} \frac{\phi(\zeta)\,dm_2(\zeta)}{z-\zeta}\Bigr|
 +\Bigl| \int_{D(z,R)} \frac{\phi(\zeta)\,dm_2(\zeta)}{z-\zeta}\Bigr| 
%\\ +\Bigl| \int_{\CC\setminus(D(0,|z|/2)\cup D(z,1))} \frac{\phi(\zeta)\,dm_2(\zeta)}{z-\zeta}\Bigr|
\\ \lesssim \frac{\|\phi\|_{L^1(\CC)}}{R}+R\cdot o(1),\qquad |z|\to\infty.
\end{multline*}
The proof of the second inequality is analogous. 

To prove the third inequality, we verify that
\begin{multline*}
\Bigl|\int_\CC \frac{\sigma_0(\zeta)\overline{F_1(\zeta)}}{z-\zeta}\,d\nu(\zeta) \Bigr| \\ 
\le \Bigl| \int_{\CC\setminus D(z,1)} \frac{\sigma_0(\zeta)\overline{F_1(\zeta)}\,dm_2(\zeta)}{z-\zeta}\Bigr|
 +\Bigl| \int_{D(z,1)} \frac{\sigma_0(\zeta)\overline{F_1(\zeta)}\,dm_2(\zeta)}{z-\zeta}\Bigr| 
%\\ +\Bigl| \int_{\CC\setminus(D(0,|z|/2)\cup D(z,1))} \frac{\phi(\zeta)\,dm_2(\zeta)}{z-\zeta}\Bigr|
\\ \lesssim 
\|F_1\|_{L^2(\CC)}\Bigl( \int_{\CC\setminus D(z,1)} \frac{dm_2(\zeta)}{(1+|\zeta|^2)|z-\zeta|^2}\Bigr)^{1/2}
 +o\Bigl(\frac{1}{|z|}\Bigr)\\=O(|z|^{-1}\log^{1/2}|z|),\qquad |z|\to\infty.
\end{multline*}
\end{proof}

Given $F\in\FF$, $z\in\mathbb C$, set
$$
\mathfrak A(F,z)(\zeta)=\frac{F(\zeta)\sigma_0(z)-F(z)\sigma_0(\zeta)}{z-\zeta}, \qquad \zeta\in\mathbb C.
$$
Then 
$$
\mathfrak A(F,z)=\sigma_0(z)\frac{F-F(z)}{z-\cdot}+F(z)\frac{\sigma_0(z)-\sigma_0}{z-\cdot}\in\FF.
$$

Given $F_1,F_2\in\FF$, set
\begin{multline*}
\mathcal I(F_1,F_2)(z)=\frac{1}{\sigma_0(z)}\langle \mathfrak A(F_2,z),F_1\rangle\\ =
\int_\CC \frac{F_2(\zeta)\overline{F_1(\zeta)}}{z-\zeta}\,d\nu(\zeta)-
\frac{F_2(z)}{\sigma_0(z)}\int_\CC \frac{\sigma_0(\zeta)\overline{F_1(\zeta)}}{z-\zeta}\,d\nu(\zeta).
\end{multline*}
Then $\sigma_0\cdot\mathcal I(F_1,F_2)\in\mathcal E$.

The following result is contained in the proof of Lemma~4.3 of \cite{BBB1}.

\begin{lemma} Let $F_1,F_2\in\FF$. Then 
$$
\mathcal I(F_1,F_2)(z)=\frac{\langle F_2,F_1\rangle}{z}+o(|z|^{-1}),\qquad |z|\to\infty,\,z\in\CC\setminus\Omega, 
$$
for some thin set $\Omega$.
\label{L15}
\end{lemma}

The system $\mathcal K=\{k_w\}_{w\in \ZZ_0}$ is a complete and minimal system in $\mathcal F$, and the system 
$\{g_w\}_{w\in \ZZ_0}$, $g_w=\sigma_0/(\sigma_0'(w)(\cdot-w))$, is biorthogonal to $\mathcal K$, see \cite{Bel,BBB1}.

Lemmas~2.3 and 4.1 of \cite{BBB1} give us the following result:

\begin{lemma} Let $F_1,F_2\in\FF$. We define 
$$
c_w=\langle F_2,k_w\rangle\cdot \langle g_w,F_1\rangle=
\frac{F_2(w)}{\sigma'_0(w)}\Bigl\langle \frac{\sigma_0}{\cdot-w},F_1\Bigr\rangle,\qquad w\in\ZZ_0.
$$
Then 
\begin{equation}
\sum_{w\in \ZZ_0}\frac{|c_w|^2}{\log(1+|w|)}<\infty,
\label{nw3}
\end{equation}
and for every $\mu\in Z(F_2)\setminus \ZZ_0$ we have  
$$
%\CCC(z)=
\sum_{w\in \ZZ_0}c_w\Bigl[ \frac1{z-w}+\frac1{w-\mu} \Bigr] 
= \mathcal I(F_1,F_2)(z) 
+\Bigl\langle \frac{F_2}{\cdot-\mu}, F_1 \Bigr\rangle,\qquad z\in\CC\setminus\ZZ_0,
$$
with the series converging absolutely in $\CC\setminus\ZZ_0$.
\label{L1}
\end{lemma}

The following lemma establishes some relations between the orthogonality in the Fock space and the corresponding discrete Cauchy transform. 

\begin{lemma} Let $F_2,F_3\in\EE_2$, $F_1,F_2F_3\in\FF$, and let $F_3$ be of completely regular growth, $\inf_{[0,2\pi]}h_{F_3}=\eta>0$. 
%, $\card Z(F)\ge 3$. 
Suppose that
$$
F_1\perp \frac{F_2F_3}{\cdot-\lambda},\qquad \lambda\in Z(F_3),
$$
and define
$$
d_w=\langle g_w,F_1\rangle=\frac{1}{\sigma'_0(w)}\Bigl\langle \frac{\sigma_0}{\cdot-w},F_1\Bigr\rangle,\qquad w\in\ZZ_0.
$$
Fix two distinct points $\lambda_1,\lambda_2\in Z(F_3)$ and set 
$$
\CCC(z)=\sum_{w\in\ZZ_0}\frac{d_wF_2(w)F_3(w)}{(z-w)(w-\lambda_1)(w-\lambda_2)}.
$$
Then for every $\ve>0$,
\begin{equation}
\CCC(z)=o(1),\qquad |z|\to\infty,\, \dist(z,\ZZ_0)\ge \ve.
\label{starT}
\end{equation}
Set $U=\sigma_0\cdot\mathcal I(F_1,F_2)$. Then $U\in \EE_{2,(\pi/2)-\eta}$ and 
\begin{gather*}
\sigma_0\cdot\CCC=\frac{UF_3}{(\cdot-\lambda_1)(\cdot-\lambda_2)},\\
U(w)=d_w\sigma'_0(w)F_2(w),\qquad  w\in\ZZ_0.
\end{gather*}
\label{L2}
\end{lemma}

\begin{proof} By Lemma~2.3 in \cite{BBB1}, we have 
$$
|d_w|\lesssim e^{-(\pi/2)|w|^2}\log^{1/2}(2+|w|), \qquad w\in\ZZ_0,
$$
and, hence,
$$
\sum_{w\in\ZZ_0}\frac{|F_2(w)F_3(w)|\cdot|d_w|}{|w|^2}<\infty.
$$
This implies \eqref{starT}.

By the simple argument in the proof of Lemma~3.1 in \cite{Bel},
%Lemma~2.4 in \cite{BBB1}, 
for every three distinct points $\lambda_1,\lambda_2,\lambda_3\in Z(F_3)$ we have 
$$
0=\Bigl\langle \frac{F_2F_3}{(\cdot-\lambda_1)(\cdot-\lambda_2)(\cdot-\lambda_3)},F_1 \Bigr\rangle=
\sum_{w\in\ZZ_0}\frac{d_wF_2(w)F_3(w)}{(w-\lambda_1)(w-\lambda_2)(w-\lambda_3)}.
$$
Hence, for fixed $\lambda_1,\lambda_2\in Z(F_3)$ we obtain 
$$
\CCC(z)=\sum_{w\in\ZZ_0}\frac{d_wF_2(w)F_3(w)}{(z-w)(w-\lambda_1)(w-\lambda_2)}=
\frac{F_3(z)U(z)}{\sigma_0(z)(z-\lambda_1)(z-\lambda_2)}
$$
for some entire function $U$. Next, since $\eta=\inf_{[0,2\pi]}h_{F_3}>0$, we have $U\in \EE_{2,(\pi/2)-\eta}$. 
Comparing the residues, we conclude that $U(w)=d_w\sigma'_0(w)F_2(w)$, $w\in\ZZ_0$.

Finally, set 
$$
T=U-\sigma_0\cdot\mathcal I(F_1,F_2).
$$
Then $T\in\EE$, and, by Lemma~\ref{L1}, $T$ vanishes on $\ZZ_0$. Set $\widetilde T=T/\sigma_0$. We have $\widetilde T\in\mathcal E$ and, by Lemma~\ref{LA},  
we obtain that $\widetilde T$ is of at most polynomial growth.
%is an entire function of finite order. 
Lemma~\ref{L15} implies that $\widetilde T=o(1)$ as $|z|\to\infty$, $z\in\CC\setminus\Omega$, for some thin set $\Omega$, and hence,
%By Lemma~\ref{L40} we conclude that 
$\widetilde T=0$, $T=0$.
\end{proof}

\begin{lemma} Let $F_1,F_2\in\FF$, and suppose that
$$
F_1\not\perp F_2, 
$$
and for some $E\in\EE$ and $P\in\PP$ we have  
$$
\mathcal I(F_1,F_2)=\frac{E}{\sigma_0}+P.
$$
Then, given $\gamma>0$, there exists a lattice thin set $\ZZ_1$ such that $E$ has at least one zero in every disc $D(w,\gamma)$, $w\in\ZZ_0\setminus \ZZ_1$.
\label{L26}
\end{lemma}

\begin{proof} 
By Lemma~\ref{L15},
\begin{equation}
\mathcal I(F_1,F_2)(z)=\frac{a+o(1)}{z},\qquad |z|\to\infty,\,z\in\CC\setminus\Omega, 
\label{39d}
\end{equation}
for some $a\not=0$ and for some thin set $\Omega$.

Set
\begin{align*}
I_1(z)&=\int_\CC \frac{F_2(\zeta)\overline{F_1(\zeta)}}{z-\zeta}\,d\nu(\zeta),\\
I_2(z)&=\int_\CC \frac{\sigma_0(\zeta)\overline{F_1(\zeta)}}{z-\zeta}\,d\nu(\zeta).
\end{align*}
so that $\mathcal I(F_1,F_2)=I_1-F_2I_2/\sigma_0$. 
By Lemma~\ref{LA}, for some $B<\infty$ we have 
\begin{equation}
|I_2(z)|^2\le B\frac{\log(2+|z|)}{1+|z|^2},\qquad z\in\CC.
\label{40d}
\end{equation}
Let $\gamma\in(0,1/2)$. Set 
$$
\ZZ_1=\Bigl\{ w\in\ZZ_0:\sup_{D(w,\gamma)\setminus D(w,\gamma/2)}
\Bigl|\frac{F_2}{\sigma_0}\Bigr|^2\ge \frac{|a|^2}{100 B \log(2+|z|)}\Bigr\}.
$$
If $w\in \ZZ_1$, then 
$$
\int_{D(w,1)}|F_2(\zeta)|^2\,d\nu(\zeta)\gtrsim \frac1{|w|^2\log(1+|w|)},
$$
and, hence, the set $\ZZ_1$ is lattice thin.

By \eqref{40d}, if $w\in \ZZ\setminus \ZZ_1$, then 
\begin{equation}
\sup_{z\in D(w,\gamma)\setminus D(w,\gamma/2)}
\Bigl|z\frac{F_2(z)}{\sigma_0(z)}I_2(z)\Bigr|\le \frac{|a|}4.
\label{41d}
\end{equation}

Next we use that $\bar\partial  I_1(z)=\pi F_2(z)\overline{F_1(z)}e^{-\pi|z|^2}$, and hence, $\bar\partial  I_1\in L^1(\CC)\cap L^\infty(\CC)\cap C^\infty(\CC)$. 
Furthermore, $\partial  I_1$ is the Beurling--Ahlfors transform \cite[Chapter 4]{AIM} of $\bar\partial  I_1\in L^2(\CC)$ and, hence, $\partial  I_1\in L^2(\CC)$.
Set
$$
\ZZ_2=\Bigl\{ w\in\ZZ_0:\int_{D(w,\gamma)\setminus D(w,\gamma/2)}
|\nabla I_1(z)|^2\,dm_2(z)\ge \frac{|a|^2\gamma^2}{100|w|^2}\Bigr\}.
$$
Since $\nabla I_1\in L^2(\CC)$, the set $\ZZ_2$ is lattice thin.
Furthermore, if $w\in \ZZ\setminus \ZZ_2$, then there exists $Q(w)\subset (\gamma/2,\gamma)$ such that 
$T(w)=\cup_{r\in Q(w)}\partial D(w,r)$ satisfies the conditions $m_2(T(w))\ge \gamma^2$ and 
$$
\osc_{T(w)}(zI_1(z))\le \frac{|a|}4.
$$
Here and later on, 
$$
\osc_A(f)=\sup_{z_1,z_2\in A}|f(z_1)-f(z_2)|.
$$

Now, if $w\in \ZZ\setminus (\ZZ_1\cup \ZZ_2)$, then, by \eqref{41d}, we have 
\begin{equation}
\osc_{T(w)}(z\mathcal I(F_1,F_2)(z))\le \frac{|a|}2.
\label{42d}
\end{equation}

Set 
\begin{multline*}
\ZZ_3=\Bigl\{ w\in\ZZ_0\setminus (\ZZ_1\cup \ZZ_2):\\
|z\mathcal I(F_1,F_2)(z)-a|\ge \frac{3|a|}4 \text{\ for some\ }z\in T(w)
\Bigr\}.
\end{multline*}
By \eqref{39d} and \eqref{42d}, the set $\ZZ_3$ is lattice thin.

Now, if $w\in\ZZ_0\setminus (\ZZ_1\cup \ZZ_2\cup \ZZ_3)$ and $r=r(w)\in Q(w)$, then 
$$
|z\mathcal I(F_1,F_2)(z)-a|\le\frac{3|a|}4,\qquad z\in \partial D(w,r).
$$
Thus the total change of the argument of $E/\sigma_0$ along $\partial D(w,r)$ is $0$ (consider first the case $P=0$, then the case $P\not=0$), and, hence, $E$ has one zero in $D(w,r)$.

Finally, $E$ has at least one zero in every disc $D(w,\gamma)$ for $w\in\ZZ_0$ outside a lattice thin set. 
\end{proof}

\section{Proof of Theorem \ref{mainthm2}} 
\label{sect4}

We start this section with four lemmas dealing with the closed polynomial span $[F]_\FF$ of $F\in\FF$. Then we pass to the proof of the theorem.

\begin{lemma} Let $F\in\EE$, and suppose that for every $A>0$, the function $F$ satisfies \eqref{expest}.
Then
$$
\overline{\Span}\bigl\{e_\lambda F: \lambda\in\mathbb C\bigr\}=[F]_\FF.
$$
\label{L3}
\end{lemma}

\begin{proof} By \eqref{expest}, we have
$$
\int_\CC |F(z)|^2\Bigl( \sum_{k\ge 0} \frac{|\lambda z|^k}{k!} \Bigr)^2\,d\nu(z)<\infty,\qquad \lambda\in\CC.
$$
Therefore,
$$
e_\lambda F\in [F]_\FF,\qquad \lambda\in\CC.
$$

In the opposite direction, let $H\in\FF$ be orthogonal to all $e_\lambda F$, $\lambda\in\CC$. Set 
$$
a_n=\int_\CC \zeta^nF(\zeta)\overline{H(\zeta)}\,d\nu(\zeta),\qquad n\ge 0.
$$
By \eqref{expest}, the series $\sum_{n\ge 0}a_nz^n/n!$ converges in the whole plane and equals to the zero function. Hence, $a_n=0$, $n\ge 0$. By the Hahn--Banach 
theorem, we conclude that 
$$
\PP F\subset \overline{\Span}\bigl\{e_\lambda F: \lambda\in\mathbb C\bigr\}.
$$
\end{proof}

\begin{lemma} Let $F\in\FF_0$, $H\in [F]_\FF$. 
Then for every $\lambda\in Z(H)\setminus Z(F)$ we have 
$$
\frac{H}{\cdot-\lambda}\in[F]_\FF.
$$
\label{L4}
\end{lemma}

\begin{proof} If $P\in\PP$ and 
$$
\|H-PF\|<\ve,
$$
then
$$
|P(\lambda)|\le \ve\cdot C(F,\lambda).
$$
Hence, 
\begin{gather*}
\|H-(P-P(\lambda))F\|\le \ve\cdot C_1(F,\lambda),\\
\Bigl\|\frac{H}{\cdot-\lambda}-\frac{P-P(\lambda)}{\cdot-\lambda}F\Bigr\|\le \ve\cdot C_2(F,\lambda).
\end{gather*}
Thus, $H/(\cdot-\lambda)$ can be approximated by elements in $[F]_\FF$.
\end{proof}

\begin{lemma} Let $F\in\FF_0$ and let $H\in\EE_2$ be of completely regular growth. Suppose that $\inf_{[0,2\pi]}h_H>0$, 
$H$ has simple zeros, $Z(F)\cap Z(H)=\emptyset$, and $FH\in[F]_\FF$. Next, let $W\in\EE$ be such that $FW\in\FF$ and 
$$
FW\perp [F]_\FF.
$$
Then 
\begin{equation}
H(z)\int_\CC \frac{F(\zeta)\overline{F(\zeta)W(\zeta)}}{z-\zeta}\,d\nu(\zeta)=
\int_\CC \frac{F(\zeta)H(\zeta)\overline{F(\zeta)W(\zeta)}}{z-\zeta}\,d\nu(\zeta).
\label{star5}
\end{equation}
\label{L5}
\end{lemma}

\begin{proof} Denote
\begin{align*}
A(z)&=\int_\CC \frac{F(\zeta)\overline{F(\zeta)W(\zeta)}}{z-\zeta}\,d\nu(\zeta),\\
B(z)&=\int_\CC \frac{F(\zeta)H(\zeta)\overline{F(\zeta)W(\zeta)}}{z-\zeta}\,d\nu(\zeta).
\end{align*}
Since $\bar\partial (AH-B)=0$, $AH-B$ is an entire function; this function vanishes on $Z(H)$ because of Lemma~\ref{L4}. 
Denote $T=A-B/H$. We have $T\in\mathcal E$. Applying Lemma~\ref{LA} to $A$ and $B$ we obtain that 
$\max_{\theta\in[0,2\pi]}|T(re^{i\theta)}|\to 0$ as $r$ tends to $\infty$, outside a set of $r$ of zero relative measure, 
and hence, $T=0$.
\end{proof}

\begin{lemma} Let $F\in\FF_0$ and let $H\in\EE_2$ be of completely regular growth. Suppose that $\eta=\inf_{[0,2\pi]}h_H>0$, 
$H$ has simple zeros, $Z(F)\cap Z(H)=\emptyset$, $(FH\cdot \EE)\cap \FF=FH\cdot\mathbb C$, and $FH\in[F]_\FF$. 
Then
\begin{equation}
[F]_\FF=\EE F\cap\FF.
%\{FE\in\FF:E\in\EE\}.
\label{star6}
\end{equation}
\label{L6}
\end{lemma}

\begin{proof} Without loss of generality, we can assume that $F$ has infinitely many zeros. 
Shifting $F$ and $H$, if necessary, by an operator $\TT_\alpha$, 
we can assume that $Z(FH)\cap\ZZ_0=\emptyset$. 
Suppose that \eqref{star6} does not hold and choose $V\in\EE\setminus\{0\}$ such that $FV\in\FF$ and 
$$
FV\perp [F]_\FF.
$$
By Lemma~\ref{L4}, 
$$
FV\perp\frac{FH}{\cdot-\lambda},\qquad \lambda\in Z(H).
$$
Set
$$
a_w=F(w)V(w),\qquad d_w=\frac1{\sigma'_0(w)}\Bigl\langle \frac{\sigma_0}{\cdot-w},FV\Bigr\rangle,\qquad w\in\ZZ_0.
$$
Fix two distinct points $\lambda_1,\lambda_2\in Z(H)$ and set $U=\sigma_0\cdot \mathcal I(FV,F)$. 
By Lemma~\ref{L2} (applied to $F_1=FV$, $F_2=F$, $F_3=H$), $U\in \EE_{2,(\pi/2)-\eta}$, and 
$$
\CCC_1(z)=\sum_{w\in\ZZ_0}\frac{d_wF(w)H(w)}{(z-w)(w-\lambda_1)(w-\lambda_2)}=
\frac{H(z)U(z)}{\sigma_0(z)(z-\lambda_1)(z-\lambda_2)}.
$$
%and $\CCC_1$ satisfies \eqref{starT}. 
Furthermore, 
$$
U(w)=d_w\sigma'_0(w)F(w),\qquad  w\in\ZZ_0,
$$
and for every $\ve>0$,
\begin{equation}
\CCC_1(z)=o(1),\qquad |z|\to\infty,\, \dist(z,\ZZ_0)\ge \ve.
\label{starT1}
\end{equation}

Fix $\mu\in Z(FV)$ and set
\begin{gather}
\CCC_2(z)=\sum_{w\in\ZZ_0}a_wd_w\Bigl[\frac1{z-w}+\frac1{w-\mu}\Bigr],\notag \\
R=\frac{UV}{\sigma_0}-\CCC_2.\label{yuyu}
\end{gather}

By Lemma~\ref{L1} (applied to $F_1=F_2=FV$), the series converges absolutely in $\CC\setminus \ZZ_0$. 
By \eqref{nw3}, for every $\ve>0$,
\begin{equation}
\CCC_2(z)=o(\log|z|),\qquad |z|\to\infty,\, \dist(z,\ZZ_0)\ge \ve.
\label{nw4}
\end{equation}

Comparing the residues, we conclude that $R\in\EE$. 

Next, choose two distinct points $\mu_1,\mu_2\in Z(F)$ and write the Lagrange interpolation formula 
$$
\CCC_3(z)=\sum_{w\in\ZZ_0}\frac{F(w)V(w)}{\sigma'_0(w)(z-w)(w-\mu_1)(w-\mu_2)}=
\frac{F(z)V(z)}{\sigma_0(z)(z-\mu_1)(z-\mu_2)}.
$$
Since $FV\in\FF$, the series converges absolutely in $\CC\setminus (\ZZ_0\cup\{\mu_1,\mu_2\})$, and for every $\ve>0$,
\begin{equation}
\CCC_3(z)=o(1),\qquad |z|\to\infty,\, \dist(z,\ZZ_0)\ge \ve.
\label{nw5}
\end{equation}

Now,
\begin{multline*}
\CCC_1(z)\CCC_3(z)=\frac{H(z)U(z)}{\sigma_0(z)(z-\lambda_1)(z-\lambda_2)}\cdot \frac{F(z)V(z)}{\sigma_0(z)(z-\mu_1)(z-\mu_2)}\\
=\frac{F(z)H(z)}{\sigma_0(z)}\cdot \frac{R(z)+\CCC_2(z)}{(z-\lambda_1)(z-\lambda_2)(z-\mu_1)(z-\mu_2)}.
\end{multline*}
By \eqref{starT1}, \eqref{nw4}, \eqref{nw5}, and the maximum principle, $FHR$ belongs to $\PP{\cdot}\FF$. 
Furthermore, $FH(R-1)\in\PP{\cdot}\FF$. 
Since the only entire multiples of $FH$ in $\FF$ are the constant ones, we conclude 
that %$FH(R-1)\in\PP\FF$, and, furthermore, that 
$R$ is a polynomial.

By Lemma~\ref{L1} (applied to $F_1=F_2=FV$) and \eqref{yuyu} we have 
$$
\frac{UV}{\sigma_0}=\mathcal I(FV,FV)+
\int_\CC \frac{|F(\zeta)V(\zeta)|^2}{\zeta-\mu}\,d\nu(\zeta)+R.
$$
%Let $\gamma>0$ be chosen later on. 
By Lemma~\ref{L26} (applied to $F_1=F_2=FV$), we obtain that
$UV$ has at least one zero in every disc $D(w,1/10)$ for $w\in\ZZ_0$ outside a lattice thin set. 
Since $R$ is a polynomial, \eqref{yuyu} and \eqref{nw4} show that 
$UV\in \PP{\cdot}\FF$. Since $U\in \EE_{2,(\pi/2)-\eta}$, a subset of $Z(V)$ of positive lower density is contained in $\cup_{w\in\ZZ_0}D(w,1/10)$. 
Repeating the above argument for $\TT_{1/2}(F)$ and $\TT_{1/2}(V)$, we obtain $U_1\in \EE_{2,(\pi/2)-\eta}$ such that $W=\TT_{1/2}(V)U_1$ has at least one zero in every disc $D(w,1/10)$ for $w\in\ZZ_0$ outside a lattice thin set. Since $W\in \PP\cdot\FF$, and a subset of $Z(W)$ 
of positive lower density is contained in $\mathbb C\setminus\cup_{w\in\ZZ_0}D(w,1/10)$, we obtain a contradiction. Thus, relation \eqref{star6} does hold. 
\end{proof}

\begin{proof}[Proof of Theorem~\ref{mainthm2}]  Suppose that
$$
\overline{\Span}\bigl\{e_\lambda F: \lambda\in\mathbb C\bigr\}\not=\EE F\cap \FF. 
%\{FE\in\FF:E\in\EE\}.
$$
Set $V(z)=G(z)\sigma((1-\alpha)^{1/2}z)$. We have $V\in\EE_{2,q}$ for some $q\ge 1$. Furthermore, $FV\in\FF$. 
Without loss of generality, we can assume that $FV$ has simple zeros. Otherwise, we can shift a bit the zeros of $F$ and $G$ without changing our 
hypothesis and conclusions. 
By Lemmas~\ref{L3} and \ref{L6}, 
$$
FV\not\in [F]_\FF.
$$
Next, let 
$$
V_s(z)=V(sz),\qquad 0<s\le 1.
$$
Since $F\in \EE_{2,(\pi/2)-\beta}$ for some $\beta>0$, $FV\in\mathcal F$, and $V$ is of completely regular growth with $\inf_{[0,2\pi]}h_V>0$, we obtain that $FV_s\in\mathcal F$, $0<s<1$. 

Let $0<\eta<\eta_1<\sqrt{\beta/q}$, $0<t\le \eta$. Let $P_n$, $n\ge 0$, be the $n$-th partial sum of the Taylor series of $V_t$. Then
$$
\sup_{z\in\CC}|P_n(z)-V_t(z)|e^{-\eta_1^2q|z|^2}\to 0,\qquad n\to\infty.
$$
Hence,
$$
P_nF\to V_tF
$$
in $\FF$ as $n\to\infty$. 
Thus, 
$$
FV_t\in [F]_\FF,\qquad 0<t\le \eta.
$$
Hence, there exist $\delta\in\bigl(0,\min(1-\eta,\eta^2(1-\alpha)/(2q))\bigr)$ and $s\in[\eta,1-\delta]$ such that 
$$
FV_s\in [F]_\FF,\quad FV_{s+\delta}\not\in [F]_\FF.
$$
Once again, without loss of generality, we can assume that $FV_s$ has simple zeros and $Z(FV_{s+\delta})\cap \mathcal Z_0=\emptyset$. 
%If necessary, we can replace $V(z)$ by $G(z)\sigma(e^{i\theta}(1-\alpha)^{1/2}z)$ for suitable $\theta$.

Choose $W\in\EE$ such that 
$$
FW\in\FF,\quad FW\perp  [F]_\FF,\quad FW\not\perp FV_{s+\delta}.
$$
By Lemma~\ref{L4}, 
\begin{equation}
FW\perp \frac{FV_s}{\cdot-\lambda},\qquad \lambda\in Z(V_s).
\label{lt}
\end{equation}
Set
$$
a_w=F(w)V_{s+\delta}(w),\qquad d_w=\frac1{\sigma'_0(w)}\Bigl\langle \frac{\sigma_0}{\cdot-w},FW\Bigr\rangle,\qquad w\in\ZZ_0.
$$
Furthermore, set 
\begin{equation}
U=\sigma_0\cdot \mathcal I(FW,F).
\label{ltt}
\end{equation}
By \eqref{lt} and by Lemma~\ref{L2} (applied to $F_1=FW$, $F_2=F$, $F_3=V_s$), we have 
$U\in \EE_{2,\pi\alpha/2}$ and 
$U(w)=d_w\sigma'_0(w)F(w)$, $w\in\ZZ_0$.

Fix $\mu\in Z(FV_{s+\delta})\setminus \mathcal Z_0$ and define 
\begin{gather}
\CCC(z)=\sum_{w\in\ZZ_0}a_wd_w\Bigl[\frac1{z-w}+\frac1{w-\mu}\Bigr],\notag\\
S=\frac{UV_{s+\delta}}{\sigma_0}-\CCC.\label{sha1}
\end{gather}
Comparing the residues we see that $S\in\EE$. 
By Lemma~\ref{L1} (applied to $F_1=FW$, $F_2=FV_{s+\delta}$), the series defining $\CCC$ converges absolutely in $\mathbb C\setminus \mathcal Z_0$, and  
$$
\frac{UV_{s+\delta}}{\sigma_0}=\mathcal I(FW,FV_{s+\delta})
+\Bigl\langle \frac{FV_{s+\delta}}{\cdot-\mu}, FW \Bigr\rangle+S.
$$
As in the proof of Lemma~\ref{L6}, we obtain that for every $\ve>0$,
\begin{equation}
\CCC_2(z)=o(\log|z|),\qquad |z|\to\infty,\, \dist(z,\ZZ_0)\ge \ve.
\label{sha2}
\end{equation}

By \eqref{ltt},
$$
\frac{UV_{s+\delta}}{\sigma_0}=V_{s+\delta}\cdot \mathcal I(FW,F).
%\frac{V_{s+\delta}(z)}{\sigma_0(z)}\int_\CC \frac{F(\zeta)\sigma_0(z)-F(z)\sigma_0(\zeta)}{z-\zeta}\overline{F(\zeta)W(\zeta)}\,d\nu(\zeta).
$$
Hence,
\begin{multline}
\label{nw7}
S(z)=-\int_\CC \frac{F(\zeta)V_{s+\delta}(\zeta)\overline{F(\zeta)W(\zeta)}}{z-\zeta}\,d\nu(\zeta)\\+
V_{s+\delta}(z)\int_\CC \frac{F(\zeta)\overline{F(\zeta)W(\zeta)}}{z-\zeta}\,d\nu(\zeta)
-\Bigl\langle \frac{FV_{s+\delta}}{\cdot-\mu}, FW \Bigr\rangle.
\end{multline}
By Lemma~\ref{L5} applied with $H=V_s$ we have 
\begin{multline*}
S(z)=-\int_\CC \frac{F(\zeta)V_{s+\delta}(\zeta)\overline{F(\zeta)W(\zeta)}}{z-\zeta}\,d\nu(\zeta)
-\Bigl\langle \frac{FV_{s+\delta}}{\cdot-\mu}, FW \Bigr\rangle\\+
\frac{V_{s+\delta}(z)}{V_s(z)}\int_\CC \frac{F(\zeta)V_s(\zeta)\overline{F(\zeta)W(\zeta)}}{z-\zeta}\,d\nu(\zeta).
\end{multline*}
Therefore, $S\in\EE_{2,q(s+\delta)^2-qs^2}\subset \EE_{2,2q\delta}$. Hence, $\widetilde S=S_{(2q\delta)^{-1/2}}\in\FF$.

Next, by \eqref{nw7} and by Lemma~\ref{LA}, 
$S$ is bounded on $Z(V_{s+\delta})$ and, hence, $\widetilde S$ is bounded on a lattice of density at least $\eta^2(1-\alpha)/(2q\delta)>1$. By Lemma~\ref{L22}, 
$\widetilde S=S$ is a constant. Thus, 
$$
\frac{UV_{s+\delta}}{\sigma_0}=\mathcal I(FW,FV_{s+\delta})
+\Bigl\langle \frac{FV_{s+\delta}}{\cdot-\mu}, FW \Bigr\rangle+S(0).
$$
By \eqref{sha1} and \eqref{sha2}, $UV_{s+\delta}\in\mathcal P\cdot\mathcal F$. 
By Lemma~\ref{L26} (applied to $FW$ and $FV_{s+\delta}$), we conclude that
$UV_{s+\delta}$ has at least one zero in every disc $D(w,1/10)$ for $w\in\ZZ_0$ outside a lattice thin set. 
However, shifting $F$ and $V_{s+\delta}$ we obtain that a subset of $Z(V_{s+\delta})$ of positive density belongs to $\CC\setminus\cup_{w\in\ZZ_0}D(w,1/10)$. We get a contradiction, which completes the proof.
\end{proof}

\section{Proof of Theorem~\ref{thm2} and Corollary~\ref{cor}}
\label{sect5}

\begin{proof}[Proof of Corollary~\ref{cor}]
Choose an integer $N$ such that 
$$
\sup_{z\in\CC}\card(Z(F)\cap Q_{z,N})<N^2,
$$
where $Q_{z,N}$ is the square centered at $z$ of sidelength $\sqrt\pi N$.
Let $M\in N\mathbb N$ be a number to be chosen later on.
Choose $\Lambda\subset\mathbb C$ disjoint from $Z(F)$ in such a way that 
$$
\card(\Gamma_z)=\alpha M^2,\qquad z\in \ZZ^M=M\mathbb Z\times M\mathbb Z,
$$
for some $\alpha<1$, where $\Gamma_z=(Z(F)\cup\Lambda)\cap Q_{z,M}$. Without loss of generality $Z(F)\cup\Lambda$ is disjoint from $\cup_{z\in \ZZ^M}\partial Q_{z,M}$.

We can choose large $M$ and place the points of $\Lambda$ in such a way that the measures 
$\alpha M^2\delta_z-\sum_{\lambda\in \Gamma_z}\delta_\lambda$, $z\in \ZZ^M$,  
have the first three moments equal to $0$. Then, arguing as in \cite[Section 4.4]{BGL} (see also \cite{LM}), we obtain that the canonical product $H$ corresponding to 
the set $Z(F)\cup\Lambda$ satisfies the estimate  
\begin{equation}
\dist(w,Z(H))^Be^{-\alpha|w|^2/2}\lesssim |H(w)|\lesssim e^{-\alpha|w|^2/2},\qquad w\in\CC,
\label{f78}
\end{equation}
for some positive number $B$. Indeed, choose $w\in\CC$, $z\in \ZZ^M$, and denote $L(\zeta)=\log(1-w/\zeta)$. 
We have $\alpha M^2=\card \Gamma_z$ and 
\begin{multline*}
\biggl| \sum_{\lambda\in\Gamma_z}\Bigl(\log\Bigl|1-\frac w{\lambda}\Bigr| -\log\Bigl|1-\frac w{z}\Bigr|\Bigr) \biggr| \\ \le 
|L'(z)|\cdot \Bigl| \sum_{\lambda\in \Gamma_z} (z-\lambda)\Bigr|+\frac12|L''(z)|\cdot \Bigl| \sum_{\lambda\in \Gamma_z} (z-\lambda)^2\Bigr|\\ +
C%M^4
\cdot \Bigl(\frac1{|w-z|^3}+\frac1{|z|^3}\Bigr).
\end{multline*}
Since two first sums on the right hand side are equal to zero, summing up by $z\in \ZZ^M$, we arrive at \eqref{f78}.

Choose $\lambda_1,\lambda_2\in\Lambda$ and set $H_1=H/[(\cdot-\lambda_1)(\cdot-\lambda_2)]$, $G=H_1/F$. Then $G$ is of completely 
regular growth, $h_G\equiv c>0$, and 
$(FG\cdot \EE)\cap \FF_{(\alpha)}=FG\cdot\mathbb C$. 
It remains to apply Theorem~\ref{mainthm2}.
\end{proof}

\begin{lemma} There exists $K\in (1,\infty)$ such that if $\eta>0$, $F\in\EE_{2,\eta}$, $V\in\EE$, $FV\in \FF$, then  
$V\in\EE_{2,(\pi/2)+K\eta}$.
\label{sha3}
\end{lemma}

\begin{proof} See \cite[Chapter I, Sections 8,9]{Lev}.
\end{proof}

\begin{proof}[Proof of Theorem~\ref{thm2}] Suppose that $\eta<\pi\beta/(8K)$, where $\beta\in(0,1)$ is the number in the statement of Lemma~\ref{L24} 
and $K$ is the number in the statement of Lemma~\ref{sha3}. 
Choose $\varepsilon\in(2K\eta/\pi,\beta/4)$ and set $H(z)=\sigma((1-\varepsilon)^{1/2}z)$. Clearly, $FH\in [F]_\FF$. Without loss of generality, 
we can assume that $Z(H)\cap Z(F)=\emptyset$. 

Suppose that the claim of the theorem does not hold. 
Choose $V\in\EE$ such that $FV\in\FF$ and $FV\perp [F]$ and set $U=\sigma_0\cdot \mathcal I(FV,F)$. 
%Arguing as in the proof of 
By Lemma~\ref{L2} (applied to $F_1=FV$, $F_2=F$, $F_3=H$), we get $U\in \EE_{2,\pi\varepsilon/2}$. Next, arguing as in the proof of Lemma~\ref{L6}, we obtain that 
$$
\frac{UV}{\sigma_0}=\mathcal I(FV,FV)+S,
$$
for some $S\in\EE$. Since $V\in\EE_{2,\pi/2+K\eta}$, we have
%standard properties of entire functions (see \cite[Section 1.8]{Lev}) show that 
$S\in\EE_{2,\pi\varepsilon}$. 

Replacing $F$ and $V$ by $F_\alpha=\mathcal T_\alpha(F)$ and $V_\alpha=\mathcal T_\alpha(V)$, correspondingly, we obtain 
$$
\frac{U_\alpha V_\alpha}{\sigma_0}=\mathcal I(F_\alpha V_\alpha,F_\alpha V_\alpha)+S_\alpha.
$$

If $S_\alpha$ and $S_{\alpha_1}$ are polynomials for some $\alpha,\alpha_1\in\CC$ such that $\alpha-\alpha_1\not\in\ZZ$, then we choose 
$\gamma>0$ such that $D(\alpha-\alpha_1,\gamma)\cap\ZZ=\emptyset$. By Lemma~\ref{L26} we obtain that both $U_\alpha V_\alpha$ 
and $U_{\alpha_1} V_{\alpha_1}$ have at least one zero in every disc $D(w,\gamma)$ for $w\in\ZZ$ outside a lattice thin set.
Then the lower density of the set 
$$
\{w\in\ZZ:D(w+\alpha,\gamma)\cap Z(V)\not=\emptyset\}
$$
is at least $3/4$. The same is true with $\alpha$ replaced by $\beta$. This contradicts the fact that the upper density of $Z(V)$ is at most $5/4$.

Thus, for some $\alpha_0$ and for every $\alpha\in\CC\setminus(\alpha_0+\mathcal Z)$, $S_\alpha$ is not a polynomial.    
We can find $\alpha\not\in\alpha_0+\mathcal Z$ %, $\delta>0$, 
and $\ZZ_1\subset\ZZ_0$ of lower density at least $1-\varepsilon$ such that $V_\alpha$ has no zeros 
in $\cup_{z\in \ZZ_1}D(z,\log^{-2}(1+|z|))$. Correspondingly, for some $\ZZ_2\subset\ZZ$ of lower density at least $1-2\varepsilon$ we obtain that 
$U_\alpha V_\alpha$ has no zeros 
in $\cup_{z\in \ZZ_2}D(z,\log^{-2}(1+|z|))$. By Lemma~\ref{LA}, $(U_\alpha V_\alpha/\sigma_0)-S_\alpha$ has at most polynomial growth. 
By the Rouch\'e theorem, for some $C>0$, $N<\infty$ we obtain that 
$$
\inf_{\partial D(z,\rho\log^{-2}(1+|z|))}|S_\alpha|<C|z|^N.
$$
for every $0<\rho<1$ and every $z\in \mathcal Z_2$. 
%for some $N<\infty$. 
Dividing $S_\alpha$ by several zeros (and possibly adding a polynomial), we arrive at the conditions of Lemma~\ref{L24}. 
Thus, we conclude that $S_\alpha$ is a polynomial. This contradiction completes the proof. \end{proof}

\section{Proof of Theorem \ref{cexTh}}
\label{sect6}

The main idea of the proof goes back to \cite{Scand}.
\smallskip

{\bf Step 1: Construction of $F$.} %Without loss of generality, we can assume that $\alpha\in(1,2)$. 
We choose $\beta$ such that
%$$1<\alpha<\beta<\gamma<\delta<\eta<\sigma<2.$$
$$1<\alpha<\beta<2.$$
Consider the function 
$$
f(z)=\exp\Bigl(\frac{\pi}{2}z^2-z^\beta\Bigr),\qquad z\in\Omega=\Bigl\{re^{i\theta}:r>0,\, |\theta|\le \frac{\pi}{4}\Bigr\},
$$
with the principal branch $z^\beta(1)=1$. The function $f$ is bounded on $\partial\Omega$.
Moreover,
\begin{align*}
\log|f(re^{i\theta})|&=\frac{\pi}2 \cos(2\theta) r^2-\cos(\beta\theta)r^\beta,\qquad re^{i\theta}\in\Omega,\\
\log|f(x+iy)|&=\frac{\pi}2 x^2 -x^\beta+O(1),\qquad x\to\infty,\,|y|\le 1,\\
\log|f(re^{\pm i\pi/4})|&=-\cos\Bigl(\frac{\pi\beta}4\Bigr)r^\beta,\qquad r>0,\\
\log|f'(re^{\pm i\pi/4})|&=-\Bigl(\cos\Bigl(\frac{\pi\beta}4\Bigr)+o(1)\Bigr)r^\beta,\qquad r\to\infty.
\end{align*}
Next, set
$$
f_1(z)=\frac1{2\pi i}\int_{\partial\Omega}\frac{f(w)\,dw}{z-w},\qquad z\in\CC\setminus\overline{\Omega}.
$$
It is well known that $f_1$ extends to an entire function. Indeed, 
$f_1$ is analytic in $\mathbb{C}\setminus\overline{\Omega}$.
Put
$$
f^R_1(z)=\frac1{2\pi i}\int_{\partial\Omega_R}\frac{f(w)\,dw}{z-w},\qquad z\in\CC\setminus\overline{\Omega_R},
$$
$$
\Omega_R=\Omega\cap\{|z|>R\}.
$$
The function $f^R_1$ is an analytic continuation of $f_1$ to $\mathbb C\setminus \overline{\Omega_R}$. 
Thus, when $R\to\infty$, 
$f_1$ extends to an entire function.
 
By the Sokhotski--Plemelj theorem we get
\begin{gather}
|f_1(re^{i\theta})|=\exp\Bigl[\frac{\pi}2 \cos(2\theta)r^2-\cos(\beta\theta)r^\beta\Bigr]+O(1),\label{star1}\\ \qquad\qquad\qquad\qquad\qquad\qquad\qquad\qquad\qquad re^{i\theta}\in\Omega,\,r\to\infty,\notag\\
|f_1(x+iy)|=\exp\Bigl(\frac{\pi}2 x^2-x^\beta\Bigr)+O(1),\quad x\to\infty,\,|y|\le 1.\label{star2}
\end{gather}
We fix $\delta$ such that $\beta<\delta<2$ and define  
$$
F(z)=f_1\bigl(e^{-i\pi/(2\delta)}z\bigr)+f_1\bigl(e^{i\pi/(2\delta)}z\bigr).
$$
%\newpage

\begin{claim} For every entire function $h$ of order 
at most $\alpha$ we have 
$$
hF\in \mathcal F.
$$
\end{claim}

\begin{proof}
By \eqref{star1},
$$
\log|F(re^{i\theta}|\le 
\begin{cases}
\dfrac{\pi}2(r^2-r^\beta)+O(1),\qquad \theta\in J,\\ \\
\dfrac{\pi}2 \cos(1/5)r^2,\qquad \theta\not\in J,
%\Bigl[-\frac{\pi}{2\delta}-\frac1{10},-\frac{\pi}{2\delta}+\frac1{10}\Bigr]\cup \Bigl[\frac{\pi}{2\delta}-\frac1{10},\frac{\pi}{2\delta}+\frac1{10}\Bigr].
\end{cases}
$$
where 
$$
J=\Bigl[-\frac{\pi}{2\delta}-\frac1{10},-\frac{\pi}{2\delta}+\frac1{10}\Bigr]\cup \Bigl[\frac{\pi}{2\delta}-\frac1{10},\frac{\pi}{2\delta}+\frac1{10}\Bigr].
$$
Therefore,
$$
\int_\CC |hF(z)|^2e^{-\pi|z|^2}\,dm_2(z)<\infty.
$$
\end{proof}
\medskip

{\bf Step 2: Key estimate.} Choose $\gamma\in(\beta,\delta)$.

\begin{claim} \label{est} Define by $\mathcal P_1$ 
the family of the polynomials $P$ such that 
\begin{equation}
\|PF\|_{\mathcal F}\le 1.\label{star3}
\end{equation}
Then for some $C>0$ %and $\gamma\in(\beta, \delta)$ 
we have 
$$
\sup_{P\in \mathcal P_1}|P(x)|\le C\exp(x^\gamma),\qquad x\ge 0.
$$
\end{claim}

\begin{proof}
The estimates \eqref{star2} and \eqref{star3} yield that 
$$
\int_{x>0,\,|y|\le 1}|P((x+iy)e^{i\pi/(2\delta)})|^2e^{-2x^\beta}\,dx\,dy\le C,\qquad P\in\mathcal P_1.
$$
In the same way,
$$
\int_{x>0,\,|y|\le 1}|P((x+iy)e^{-i\pi/(2\delta)})|^2e^{-2x^\beta}\,dx\,dy\le C,\qquad P\in\mathcal P_1.
$$
By the Fubini theorem, for every $P\in\mathcal P_1$ we can find $y(P)\in[-1,1]$ such that 
$$
\int_0^\infty|P((x\pm iy(P))e^{\pm i\pi/(2\delta)})|^2e^{-2x^\beta}\,dx\le C_1.
$$
Since the point evaluations are locally uniformly bounded in the Fock space, by the maximum principle we obtain that
$$
\sup_{P\in \mathcal P_1,\, |z|\le 2}|P(z)|\le C_2.
$$

Note that the lines $\{ (x + iy(P) ) e^{i\pi/(2\delta)} :\, x\in\mathbb{R}\}$ 
and $\{ (x - iy(P) ) e^{-i\pi/(2\delta)} :\, x\in\mathbb{R}\}$ 
intersect at the point $-y(P)/\sin\frac{\pi}{2\delta}$.
Therefore, if we set 
$$
Q(z) =P\Bigl(z-\frac{y(P)}{\sin\frac{\pi}{2\delta}}\Bigr),
$$
then we have
$$
\int_0^\infty|Q( t e^{\pm i\pi/(2\delta)})|^2 e^{-2t^\beta}\,dt\le C_3.
$$
Put
$$
Q_1(re^{i\theta}) = Q(r^{1/\delta}e^{i\theta/\delta})\exp\Bigl[-\frac12 r^{\gamma/\delta}e^{i\theta\gamma/\delta}\Bigr],
\qquad r\ge 0,\,|\theta|\le\frac\pi2.
$$
Then $Q_1$ is bounded and analytic in the right half-plane, and 
$$
\int_{\mathbb R} |Q_1(iy)|^2\, dy\le C_4.
$$
Therefore,
$$
|Q_1(x)|\le C_5,\qquad x\ge 1,
$$
and as a result,
$$
|P(x)|\le C_6\exp(x^{\gamma}),\qquad x\ge 0.
$$
\end{proof}
\medskip

{\bf Step 3: Construction of $G$.}
Next we fix $\sigma$ and $\eta$ such that $\delta<\eta<\sigma<2$. We consider the function 
$$
g(z)=\exp(z^\sigma),\qquad z\in\Omega_1=\Bigl\{re^{i\theta}:r>0,\, |\theta|\le \frac{\pi}{2\eta}\Bigr\},
$$
with the principal branch $z^\sigma(1)=1$.
Then
\begin{align*}
\log|g(re^{i\theta})|&=\cos(\sigma\theta)r^\sigma,\qquad re^{i\theta}\in\Omega_1,\\
\log|g(x)|&=x^\sigma,\qquad x\ge 0,\\
\log|g(re^{\pm i\pi/(2\eta)})|&=\cos\Bigl(\frac{\pi\sigma}{2\eta}\Bigr)r^\sigma,\qquad r>0,\\
\log|g'(re^{\pm i\pi/(2\eta)})|&=\Bigl(\cos\Bigl(\frac{\pi\sigma}{2\eta}\Bigr)+o(1)\Bigr)r^\sigma,\qquad r\to\infty.
\end{align*}

Set
$$
G(z)=\frac1{2\pi i}\int_{\partial\Omega_1}\frac{g(w)\,dw}{z-w},\qquad z\in\CC\setminus\overline{\Omega_1}.
$$
Then $G$ extends to an entire function, 
\begin{equation}
|G(x)|=\exp(x^\sigma)+O(1),\qquad x\ge 0,
\label{star1u}
\end{equation}
and
\begin{equation}
|G(re^{i\theta})|=
\begin{cases}
\exp\Bigl(\cos(\sigma\theta)r^\sigma\Bigr)+O(1),\quad \theta\in [-\frac{\pi}{2\sigma},\frac{\pi}{2\sigma}],\\
\,O(1),\quad \theta\in [-\pi,\pi]\setminus [-\frac{\pi}{2\sigma},\frac{\pi}{2\sigma}].
\end{cases}
\label{star2u}
\end{equation}

%\newpage

\begin{claim} $FG\in \mathcal F$.
\end{claim}

\begin{proof}
By \eqref{star1} and \eqref{star2u} we have
\begin{multline*}
\log|(FG)(re^{i\theta})|-\frac{\pi}2r^2\\ \le 
\chi_{[\pi/(2\delta)-\pi/4,\pi/(2\delta)+\pi/4]}(\theta)
\Bigl[
\cos\Bigl(2\theta-\frac{\pi}{\delta}\Bigr)\cdot\frac{\pi}2r^2-\cos\Bigl(\beta\theta-\frac{\beta\pi}{2\delta} \Bigr)\cdot r^\beta
\Bigr]^+
\\+\chi_{[-\pi/(2\delta)-\pi/4,-\pi/(2\delta)+\pi/4]}(\theta)\Bigl[\cos\Bigl(2\theta+\frac{\pi}{\delta}\Bigr)\cdot\frac{\pi}2r^2-\cos\Bigl(\beta\theta+\frac{\beta\pi}{2\delta} \Bigr)\cdot r^\beta\Bigr]^+\\
+\chi_{[-\pi/(2\sigma),\pi/(2\sigma)]}(\theta)\bigl(\cos(\sigma\theta)r^\sigma\bigr)-\frac{\pi}2r^2+O(1),\qquad r\to\infty,\,\theta\in[-\pi,\pi].
\end{multline*}
Hence, for some $\ve=\ve(\sigma,\delta)>0$, $d=d(\beta,\sigma,\delta)>0$, we have 
\begin{multline*}
\log|(FG)(re^{i\theta})|-\frac{\pi}2r^2 \\ \le 
\begin{cases}
\, -dr^\beta,\quad \theta\in J=\Bigl[\frac{\pi}{2\delta} -\ve,\frac{\pi}{2\delta} +\ve\Bigr]\cup \Bigl[-\frac{\pi}{2\delta} -\ve,-\frac{\pi}{2\delta} +\ve\Bigr],\\
\, -dr^2,\quad \theta\in[-\pi,\pi]\setminus J,
%\Bigl(\Bigl[\frac{\pi}{2\delta} -\ve,\frac{\pi}{2\delta} +\ve\Bigr]\cup \Bigl[-\frac{\pi}{2\delta} -\ve,-\frac{\pi}{2\delta} +\ve\Bigr]\Bigr),
\end{cases} 
\end{multline*}
and, then, $FG\in \mathcal F$. 
\end{proof}
\medskip

{\bf Step 4: End of the proof.}
Now, we argue as in \cite{Scand}. Suppose that $P_n$ are polynomials such that
$$
P_nF\overset{\mathcal F}{\to}FG.
$$
Then for some $C_1>0$ we have  
$$
\{P_n/C_1\}_{n\ge 1}\in\mathcal P_1,
$$
and by Claim~\ref{est} we get 
\begin{equation}
|P_n(x)|\le CC_1\exp(x^\gamma),\qquad n\ge1,\,x\ge 0.
\label{44d}
\end{equation}
Since $P_nF$ tends to $FG$ uniformly on compact subsets of the complex plane, \eqref{44d} contradicts \eqref{star1u}, and  
$$
FG\not\in\clos_{\mathcal F}\{\mathcal PF\}.
$$

\end{document}